\documentclass[]{interact}

\usepackage{graphicx,amsthm,amssymb,amsmath,tikz}
\usepackage{bm}
\usepackage{enumitem}
\usepackage{hyperref}
\usepackage[toc,page]{appendix}
\usepackage{color}
\usepackage{hyperref}
\usepackage{url}
\usepackage{breakurl}
\usepackage{comment}

\definecolor{applegreen}{rgb}{0.55, 0.71, 0.0}
\definecolor{ao(english)}{rgb}{0.0, 0.5, 0.0}
\definecolor{bondiblue}{rgb}{0.0, 0.58, 0.71}
\newcommand{\bburl}[1]{\textcolor{blue}{\url{#1}}}

\newcommand{\bbn}{\mathbb{N}}

\newcommand{\bbz}{\mathbb{Z}}
\newcommand{\ds}{\displaystyle}

\newcommand{\eps}{\varepsilon}

\numberwithin{equation}{section}
\theoremstyle{plain}
\newtheorem{dfn}{Definition}[section]
\newtheorem{thm}[dfn]{Theorem}
\newtheorem{cor}[dfn]{Corollary}
\newtheorem{lemma}[dfn]{Lemma}
\newtheorem{conj}{Conjecture}
\newtheorem*{no-num-thm}{Theorem}

\theoremstyle{definition}

\theoremstyle{remark}

\newtheorem*{eg*}{Example}
\title{Sum of Consecutive Terms of Pell and Related Sequences}
\author{
\name{Navvye Anand\textsuperscript{a}\thanks{CONTACT Navvye Anand. Email: navvye.anand@caltech.edu}} \name{Amit Kumar Basistha\textsuperscript{b}\thanks{CONTACT Amit Kumar Basistha. Email: basisthaamitkumar2@gmail.com}} \name{Kenny B. Davenport\textsuperscript{c}} \name{Alexander Gong\textsuperscript{d}\thanks{CONTACT Alexander Gong. Email: ag4668@columbia.edu}} \name{Florian Luca\textsuperscript{e}\thanks{CONTACT Florian Luca. Email: Florian.Luca@wits.ac.za}}
\name{Steven J. Miller\textsuperscript{f}\thanks{CONTACT Steven J. Miller. Email: sjm1@williams.edu}} \name{Alexander Zhu\textsuperscript{g}\thanks{CONTACT Alexander Zhu. Email: lazostaf@gmail.com}}
}

\date{Summer 2023}

\begin{document}

\maketitle

\begin{abstract}
We study new identities related to the sums of adjacent terms in the Pell sequence,
defined by $P_{n} := 2P_{n-1}+P_{n-2}$ for $ n\geq 2$ and $P_{0}=0, P_{1}=1$, and generalize these identities for many similar sequences.
We prove that the sum of $N>1$ consecutive Pell numbers is a fixed integer
multiple of another Pell number if and only if $4\mid N$. We consider
the generalized Pell $(k,i)$-numbers defined by $p(n) :=\ 2p(n-1)+p(n-k-1) $ for $n\geq k+1$, with $p(0)=p(1)=\cdots =p(i)=0$ and $p(i+1)=\cdots = p(k)=1$
for $0\leq i\leq k-1$, and prove that the sum of $N=2k+2$ consecutive
terms is a fixed integer multiple of another term in the sequence. We also prove that for the generalized Pell $(k,k-1)$-numbers such a relation does not exist when $N$ and $k$ are odd. We give analogous results for the Fibonacci and other related second-order recursive sequences.
\end{abstract}


\section{Introduction}
We first review some standard notation, and then describe our results.
The \emph{Fibonacci numbers} are defined by
\begin{equation}\label{def:fibonacci}
    F(n) \ :=\ \left\{\begin{array}{cc}
     0 & n=0  \\
     1 & n=1  \\
     F(n-1)+F(n-2) &\ n\geq 2,\\
\end{array}\right.
\end{equation}
and have a closed form given by \emph{Binet's formula}
\begin{equation}\label{eqn: Binnet}
    F(n)\ =\ \frac{\varphi^n-\psi^n}{\sqrt{5}},
\end{equation}
where
\[\varphi\ =\ \frac{1+\sqrt{5}}{2}\qquad \text{and}\qquad
\psi\ =\ \frac{1-\sqrt{5}}{2}\ =\ -\frac{1}{\varphi}.\]

This is the simplest depth two constant-coefficient recurrence to study (the simpler depth one recurrences are just pure geometric sequences). They satisfy numerous interesting identities and arise in various areas; see for example \cite{tk}. We consider the Fibonacci numbers and other recursively defined sequences of numbers. In particular, we are interested in \emph{Pell numbers} and their generalizations; we state these sequences and then our results.

\begin{dfn}\label{dfn:pell}
\emph{Classical Pell numbers} are defined by the following recurrence and initial conditions:
\begin{equation}\label{rec:pell}
P(n)\ :=\ \left\{\begin{array}{cc}
     0 & n=0  \\
     1 & n=1  \\
     2P(n-1)+P(n-2) &\ n\geq 2.\\
\end{array}\right.
\end{equation}
\end{dfn}
\begin{dfn}\label{dfn:pell-lucas}
    The \emph{Pell-Lucas sequence} or the \emph{Companion Pell sequence} is defined by
    \begin{equation}
Q(n)\ :=\ \left\{\begin{array}{cc}
     2 & n=0  \\
     2 & n=1  \\
     2Q(n-1)+Q(n-2) &\ n\geq 2.\\
\end{array}\right.
\end{equation}
\end{dfn}
\begin{dfn}\label{dfn:Lucas}
    The \emph{Lucas sequence} is defined by
    \begin{equation}
L(n)\ :=\ \left\{\begin{array}{cc}
     2 & n=0  \\
     1 & n=1  \\
     L(n-1)+L(n-2) &\ n\geq 2.\\
\end{array}\right.
\end{equation}
\end{dfn}

\begin{dfn}\label{BinetFormula}
Every sequence defined by a homogeneous linear recurrence with constant coefficients has a closed-form expression. If the terms of such a sequence can be expressed in a closed form like \eqref{eqn: Binnet}, this is called \emph{a generalized Binet formula} (see \cite{BBILMT, cl}).
\end{dfn}

\begin{eg*}
Let
\begin{equation}
\label{eqn:a-b}
a\ :=\ 1+\sqrt{2}\qquad \text{and}\qquad b\ :=\ 1-\sqrt{2}\ =\ -\frac{1}{a}.
\end{equation}
Then the $n$\textsuperscript{${\rm th}$} Pell Number is given by the generalized Binet formula:
\begin{equation}
\label{P_n}
P(n)\ =\ \frac{a^n-b^n}{2\sqrt{2}}.
\end{equation}
\end{eg*}

\begin{dfn}
  Throughout the paper, we let  $C: \mathbb{N} \to \mathbb{N}$ denote a natural valued function on natural numbers.
\end{dfn}

\subsection{Motivation and Results}

We analyze the relationships between sums of consecutive numbers in recurrence sequences. The first theorem below is a generalization of an observation made by the third named author to the sixth named author (for a problem for the Pi Mu Epsilon Journal) for sums of eight consecutive terms.

\begin{thm}\label{thm:sum-4N-pell}
For any $N\in\bbn$, the sum of $4N$ consecutive Pell numbers is equal to a constant (depending on $N$) multiplied by the $(2N+1)$\textsuperscript{{\rm st}} term of the consecutive terms. In particular, we have
\begin{equation}
\label{eqn:sum-4N-pell}
\sum\limits_{i=0}^{4N-1} P(n+i)\ =\ \frac{(a^{2N}-b^{2N})}{\sqrt{2}} P(n+2N),
\end{equation}
where \[a\ =\ 1+\sqrt{2}\qquad \text{and}\qquad b\ =\ 1-\sqrt{2}\ =\ -\frac{1}{a}.\]
\end{thm}
\begin{proof}

The $n$\textsuperscript{${\rm th}$} Pell Number is given by \eqref{P_n}:
\[P(n)\ =\ \frac{a^n-b^n}{2\sqrt{2}}.\]
Therefore, we have
\begin{align}
\nonumber
\sum\limits_{i=0}^{4N-1}P(n+i)\ &=\ \sum\limits_{i=0}^{4N-1}\frac{a^{n+i}-b^{n+i}}{2\sqrt{2}}\\
\label{sum-8-proof-1}
&=\ \frac{a^n}{2\sqrt{2}}\sum\limits_{i=0}^{4N-1}a^i-\frac{b^n}{2\sqrt{2}}\sum\limits_{i=0}^{4N-1}b^i\\
\nonumber
&=\ \frac{a^n}{4}(a^{4N}-1)-\frac{b^n}{4}(1-b^{4N})\\
\nonumber
&=\ \frac{a^{n+2N}}{4}(a^{2N}-b^{2N})-\frac{b^{n+2N}}{4}(a^{2N}-b^{2N})\\
\nonumber
&=\ \frac{(a^{2N}-b^{2N})}{4}(a^{n+2N}-b^{n+2N})\\
\nonumber
&=\  \frac{(a^{2N}-b^{2N})}{\sqrt{2}} P(n+2N).
\end{align}
\end{proof}
The above result motivates the question: \emph{For which numbers $n \in \mathbb{N}$ does the sum of $n$ consecutive Pell numbers equal a fixed integer multiple of another Pell number?}\\

We answer this question for the Pell, Fibonacci and other related sequences. In particular, for the Pell sequence we observe that multiples of $4$ (see Theorem \ref{thm:sum-4N-pell}) and the trivial case of $N=1$ are the only values of $N$ that work. We then extend our methods to generalized Pell numbers and present a conjecture regarding when the sum of consecutive generalized Pell numbers equals a fixed integer multiple of another generalized Pell number. Additionally, we describe several interesting properties of Pell numbers using tilings of an $n \times 1$ board with polyominoes.


\section{Identities and Preliminary Results}
The following lemmas describe standard identities relating Pell, Lucas and Fibonacci sequences, and are used extensively in the rest of the paper.
These have been derived from the \cite{OEIS} and \cite{tk}. We provide the proofs in the appendix for completeness.

\begin{lemma}\label{lem:pell-(4N+2)}
For any non-negative integer $k$ the Pell numbers satisfy
\begin{equation}\label{eqn:pell-(4N+2)}
P(n+k)+(-1)^kP(n-k)\ =\ Q(k)P(n),
\end{equation}
where $Q(k)$ is the $k$\textsuperscript{{\rm th}} term of the Pell-Lucas sequence given in Definition \ref{dfn:pell-lucas}.
\end{lemma}
\begin{proof}
    For a proof, see Appendix \ref{subsec:proof_lem:pell-(4N+2)}.
\end{proof}

\begin{lemma}\label{lem:fib-(4N)} For the Fibonacci and Lucas numbers we have
\begin{equation}\sum\limits_{i=0}^{4N-1}F(n+i)\ =\ F(2N)L(n+2N+1).\end{equation}
\end{lemma}
\begin{proof}
    For a proof, see Appendix \ref{subsec:proof_lem:fib-(4N)}
\end{proof}
 \begin{lemma}\label{lem:sum-fib}
For the Fibonacci numbers, we have for all positive integers $n$
\begin{equation*}
\varphi^{n}\ =\ F(n)\varphi+F(n-1) \quad \text{ where } \ \quad  \varphi \ = \ \frac{1+\sqrt{5}}{2}.
\end{equation*}
    \end{lemma}

    \begin{proof}
        We proceed by induction on $n$. When $n=1$, the statement of the lemma is $\varphi=\varphi$, which is trivially true. Similarly, $n=2$ is true because $\varphi^2=\varphi+1$ is true as $\frac{1+\sqrt{5}}{2}$ is a root of the characteristic polynomial $x^2 - x - 1 = 0$. Thus we may assume the statement holds for all natural numbers less than $n \geq 3$, and show it holds for $n$:
        \begin{align*}
            \varphi^{n}\ &=\ \varphi^{n-1}+\varphi^{n-2}\\
            &=\ (F(n-1)\varphi + F(n-2)) + (F(n-2)\varphi+  F(n-2))\\
            &=\ (F(n-1)+F(n-2))\varphi+(F(n-2)+F(n-3))\\
            &=\ F(n)\varphi+F(n-1),
        \end{align*}
        which completes the proof.
    \end{proof}

\section{Some General Results}

One of our main goals is to determine not only when the sum of consecutive terms in a recurrence is a fixed multiple of a term of the recurrence, but further to determine which term. In this section we consider the following sequence.
\begin{dfn}
    Let $r$ be a non-negative integer. Consider a sequence $\{f(n)\}$ of non-negative integers recursively defined by $$ f(n) \ := \  rf(n-1) + f(n-2)$$ with initial conditions so that it is not identically zero (we call this a non-degenerate sequence).
\end{dfn}

Not only is this a natural generalization of the Fibonacci numbers (which are just the $r=1$ case), but similar to how the Fibonacci numbers count various objects, this sequence as well has a combinatorial interpretation. In \cite{DHW} the authors show that $f(n)$ is the number of $r$-regular words over $\{1, 2, \dots, n\}$ avoiding the patterns $122$ and $213$ (this means we cannot form a sub-word with three objects with this relative ordering). Additionally, $\left\{f(n)\right\}$ also makes a surprising appearance in elliptic curve research. Recent work in \cite{PiWa} shows under certain circumstances, there exists a bijection between the set of integral points on elliptic curves of the form $y^2 = (r^2 + 4)x^4 - 4$ and the set of squares in $\left\{f(n)\right\}$ with odd indices.

If we set $$\alpha\ :=\ \frac{r+\sqrt{r^2+4}}{2} \qquad \text{ and } \qquad \beta\ :=\ \frac{r-\sqrt{r^2+4}}{2} $$ then the generalized Binet formula (see \cite{BBILMT, cl})) yields that there exist complex numbers $\alpha$ and $\beta$ such that $f(n) \ =\  a\alpha^n + b\beta^n.$

\begin{thm}\label{linear relations}  
Fix any integer $N > 0$. If there is an integer $C(N)$ such that for every sufficiently large $n$ there exists an integer index $\, j(n;N)\,$ such that the following equation holds:
\[\sum\limits_{i=0}^{N-1}f(n+i)\ =\ C(N) \cdot f(j(n;N)),\]
then there is an integer $k(N)$ such that \[j(n;N) \ = \ n + k(N) \qquad \text{ and } \qquad k(N) \in \left[\frac{N}2 , N\right]. \]
\end{thm}

  \begin{proof} Define $$ b\ :=\ \frac{\alpha^N-1}{C(N) (\alpha-1)} \qquad \text { and } \qquad k(N) \ := \ \log\limits_{\alpha}b,$$ with $\alpha,\  \beta$ and $f$ as above.

     Note that $|\alpha|\ >\ |\beta|$ and $|\beta|\ <\ 1$. Then by the generalized Binet's formula $$f(n)\ =\ a\alpha^n + b\beta^n$$ and
     \begin{align*}\lim_{n\to\infty}|f(n)-a\alpha^n|\ =\ 0.
     \end{align*}
     This implies that for any $\eps > 0$ there exists a natural number $M$ such that for all $n>M$,
     \begin{align}
         |f(n)-a\alpha^n|\ <\ \frac{C(N) \cdot \eps}{2N}.
     \end{align}
     We choose  $M$ sufficiently large such that
     \begin{align}
         |f(j(n;N))-a\alpha^{j(n;N)}|\ <\ \frac{\eps}{2}
     \end{align}
     for all $n>M$. \\  \\
     \ Then
    \begin{equation}\label{bound}
    \begin{aligned}[b]
    &\left|a\alpha^{j(n;N)}-\frac{1}{C(N)}\sum\limits_{i=0}^{N-1}a\alpha^{n+i}\right| \\
    & \ <\ \left|a\alpha^{j(n;N)}-f(j(n;N))\right|+\frac{1}{C(N)}\left|C(N) \cdot f(j(n;N))-\sum\limits_{i=0}^{N-1}a\alpha^{n+i}\right|\\
    & \ <\ \frac{\eps}{2}+\frac{1}{C(N) }\left|\sum\limits_{i=0}^{N-1}f(n+i)-\sum\limits_{i=0}^{N-1}a\alpha^{n+i}\right|\\
    & \ <\ \frac{\eps}{2}+\frac{\eps}{2}\ =\ \eps.\\
    \end{aligned}
    \end{equation}

    We now have
    \begin{equation*}
        \frac{1}{C(N)}\sum\limits_{i=0}^{N-1}\alpha^{n+i}\ =\ \alpha^{n+\log_{\alpha}b}.
    \end{equation*}
   If $k(N)\not\in\bbn$ then consider $m=\min\{ n+ k(N), j(n;N)\}$. The conditions on $f$ imply that it is an increasing sequence, therefore $j(n;N)\to\infty$ as $n \to\infty$. Hence $m\to\infty$ as $n\to\infty$. We also note that

    \begin{align}\label{eqn:bound2}
       \left|a\alpha^{j(n;N)}-\frac{1}{C(N)}\sum\limits_{i=0}^{N-1}a\alpha^{n+i}\right|&\ =\  \left|a\right|\alpha^m\left|\alpha^{\left| j(n;N)-n-k(N)\right|}-1\right|\nonumber\\
       &\ \geq \ \left|a\right|\alpha^m\left(\alpha^{\left| j(n;N)-n-k(N) \right|}-1\right),
    \end{align}
    and since $j(n;N)\in \bbn$ we have $j(n;N)-n-k(N) \not\in\bbz$. Similarly, since $\alpha>1$ we have $\alpha^{\left| j(n;N)-n-k(N)\right|}-1>0$. Lastly, $m\to\infty$ as $n\to\infty$ and thus for large enough $n$, the left hand side of \eqref{eqn:bound2} tends to infinity:
    \begin{equation*}
       \lim_{n\to\infty} \left|a\alpha^{j(n;N)}-\frac{1}{C(N)}\sum\limits_{i=0}^{N-1}a\alpha^{n+i}\right|\ \to\ \infty,
    \end{equation*}
which contradicts \eqref{bound}, implying that $k(N) \in\bbn$, and $j(n;N) = n + k(N)$.

\ \\
We now prove that
    $$k(N) \in \left[ \frac{N}{2}, N\right].$$

We begin by noting that
   \begin{equation}\label{integer_constant}
    \frac{\alpha^N-1}{C(N)(\alpha-1)}\ =\
    \alpha^{k(N)}\ \implies\ C(N) \ =\ \sum\limits_{i=0}^{N-1}\alpha^{i-k(N)}.
   \end{equation} Let $k(N) < \frac{N}2$, then
    \begin{equation*}
         C(N) \ =\ 1 + \sum\limits_{i=1}^{k(N)}\left( \alpha^i + \frac{1}{\alpha^i} \right) + \sum\limits_{i=2k(N) +1}^{N-1}\alpha^{i-k(N) }.
    \end{equation*}
    Note that the coefficient of the irrational part of $\sum\limits_{i=2k(N)+1}^{N-1}\alpha^{i-k(N)}$ is a positive integer. We now have
    \begin{equation}
        \alpha^i + \frac{1}{\alpha^i}\ =\ \alpha^i + (-\beta)^i.
    \end{equation}
    Applying the binomial theorem to the above-mentioned equation gives the coefficient of the irrational part in $\alpha^i + (-\beta)^i$ for $i>1$ to be
    \begin{equation*}
        \sum\limits_{j=1}^{\lfloor \frac{i-1}{2}\rfloor}(r^2+4)^{}(r^{i-2j+1} + (-r)^{i-2j+1}) \ \geq\ 0,
    \end{equation*}
    which implies that $C(N)$ is irrational, resulting in a contradiction. Therefore $k(N) \geq N/2$. \\ \\ Now, by induction we get the following inequality:
    \begin{equation*}
         \sum\limits_{i=0}^{n+N-1}f(i) \ <\ f(n+N+1),
    \end{equation*}
    which proves $k(N)\leq N$.

This implies that $$j(n;N) \in \left[n + \frac{N}2, n + N \right].$$
\end{proof}

\begin{thm}\label{thm:non_degenerate_general_4}
    Given a non-degenerate sequence of non-negative integers recursively defined by
\[f(n)\ :=\ rf(n-1)+f(n-2), \]
where $r\in\bbn$, if
\begin{equation*}
   \sum\limits_{i=0}^{3}f(n+i)\ =\ C(N) \cdot f(j(n;N))
\end{equation*}
then $r=2$.
\end{thm}
\begin{proof}For a proof, see Appendix \ref{subsec:proof_specialized_linear_relations}.
\end{proof}

Clearly, only looking at rational multiples of terms in the sequence is sufficient, because the desired multiple can be written as a ratio of integers. The following theorem proves why only looking at \emph{integer} multiples of terms in the sequence is sufficient.

\begin{thm}
    Define $\{f(n)\}$ by the recurrence relation
    \begin{equation*}
        f(n)\ :=\ rf(n-1)+f(n-2)
    \end{equation*} where $r\in\bbn$, and choose initial conditions so that $f$ is not identically zero. Then if the sum of $N>1$ consecutive terms of $\{f(n)\}$ is a fixed rational constant times another term in the sequence, then the rational constant is an integer.
\end{thm}

\begin{proof}
  Let $d=\gcd(f(0),f(1))$. We notice that $d\mid f(n)$  for all $n \in \mathbb{N}$ and therefore consider the equivalent sequence $h(n) := f(n)/d$ instead.
  Then $$\gcd(h(n),h(n+1))\ =\ 1$$  for all $n\geq 1$ since
  \begin{align*}  \gcd(h(n),h(n+1))\ &=\ \gcd(h(n),rh(n)+h(n-1))\\
  &=\ \gcd(h(n),h(n-1))\\
  &\qquad \vdots\\
  &=\ \gcd(h(0),h(1)).\\
  \end{align*}
  Now suppose
  \begin{align}\label{eqn:rational-mult-1}
    \sum\limits_{i=0}^{N-1}f(n+i)\ =\ \frac{a}{b}f(j(n;N)),
  \end{align}
  where $a,b\in\bbz$ and $\gcd(a,b)=1$. Dividing both sides by $d$, we get
  \begin{align}\label{eqn:rational-mult-2}
    \sum\limits_{i=0}^{N-1}h(n+i)\ =\ \frac{a}{b}h(j(n;N)).
  \end{align}
  From Theorem \ref{linear relations} we know that there exists $M\in\bbn$ such that $j(n;N)=n+k(N)$ for $n>M$. Applying \eqref{eqn:rational-mult-1} tells us that $b\mid h(n+k(N)) \ \mbox{for all } n>M$. However, if $b>1$ we reach the contradiction that $\gcd(f(n),f(n+1)) \neq 1 $ for all $ n \in \bbn$. Therefore $b=1$, which completes our proof.
  \end{proof}

\section{Pell numbers}

In Theorem \ref{thm:sum-4N-pell}, we proved that the sum of $4N$ consecutive Pell numbers is a constant integer multiple of the $(2N+1)$\textsuperscript{st} term. We generalize to other related partial sums.

\subsection{Sum of $4N+2$ Consecutive Terms}

\begin{thm}\label{thm:pell-(4N+2)}
Let $P(n)$ denote the $n$\textsuperscript{${\rm th}$} Pell number. Fix any integer $N > 0$. There is no integer $C(N)$ such that for every $n$ there exists an integer index $j(n;N)$ such that the following equation holds: $$\sum_{i=0}^{2N} P(n+i)\ =\ C(N) P(j(n;N)).$$


\end{thm}

We prove this theorem in greater generality. Consider sequences  $f(n)$ satisfying the recurrence relation
    \begin{equation*}
    f(n)\ :=\ rf(n-1)+f(n-2),
\end{equation*} such that $f(0)=0$ and $f(1)=1$ and $r \geq 2$.

\begin{thm} Let $f$ be as above, and fix any integer $N > 0$. There is no integer $C(N)$ such that for every $n$ there exists an integer index $\, j(n;N)\,$ such that the following equation holds: $$\sum_{i=0}^{4N+1} f(n+i)\ =\ C(N) f(\, j(n;N)\,).$$
\end{thm}

\begin{proof}
Define the sequence $g(n)$ by $g(0)=2$, $g(1)=r$ and
\begin{equation}
    g(n)\ :=\ rg(n-1)+g(n-2).
\end{equation}
Using induction on $k$, we find
\begin{equation}
f(n+k)+(-1)^{k}f(n-k)\ =\ g(k)f(n).
\end{equation}

Using \eqref{gseqsum} we get
\begin{equation}
\sum\limits_{k=0}^{n}f(n)\ =\ \frac{1}{r}\Big(f(n) + f(n+1) - 1\Big).
\end{equation}
Therefore, we have
\begin{equation*}
\begin{split}
\sum\limits_{k=0}^{4N+1}f(n+k)\ &=\ \sum\limits_{k=0}^{n+4N+1}f(k)-\sum\limits_{k=0}^{n-1}f(k)\\
&=\ \frac{1}{r}\Big[f(n+4N+1)+f(n+4N+2)-f(n-1)+f(n)\Big]\\
&=\ \frac{1}{r}\Big[g(2N+1)f(n+2N+1)+g(2N+1)f(n+2N)\Big]\\
&=\ g(2N+1)\frac{f(n+2N+1)+f(n+2N)}{r}.
\end{split}
\end{equation*}
Further by induction on $N$ we know that $\dfrac{g(2N+1)}r\in\bbn$.\\ \

Now suppose the sum of $4N+2$ terms is a fixed multiple of another term. Then for some $t_1\geq t_2\in \mathbb{N}$, the following equations hold:
\begin{eqnarray}
sf(t_1) & \ =\ & g(2N+1)\frac{f(n+2N+2)+f(n+2N+1)}{r} \nonumber\\
sf(t_2) & \ =\ & g(2N+1)\frac{f(n+2N+1)+f(n+2N)}{r}.
\end{eqnarray}
Dividing both sides yields
\begin{equation}
\frac{f(n+2N+1)+f(n+2N+2)}{f(n+2N)+f(n+2N+1)}\ =\ \frac{f(t_1)}{f(t_2)}.
\end{equation}
Now for $m$ a positive integer let
\begin{equation*}
    T_m\ :=\ f(m)+f(m-1).
\end{equation*} Then $T_m$ satisfies the following recurrence:
\begin{equation*}
\begin{split}
rT_{m-1}+T_{m-2}\ &=\ rf(m-1)+rf(m-2)+f(m-2)+f(m-3)\\
&=\ f(m)+f(m-1)\\
&=\ T_m.
\end{split}
\end{equation*}
Using this recursion and induction it follows that
\begin{equation}
r\ <\ \frac{T_{m+1}}{T_m}\ \leq\ \frac{3r}{2},
\end{equation} which implies that $t_1>t_2$. If $t_1\geq t_2+2$ then
\begin{equation}
\begin{split}
    \frac{f(t_2+2)}{f(t_2)}\ &=\ \frac{rf(t_2+1)+f(t_2)}{f(t_2)}\\
    &=\ \frac{r^2f(t_2)+rf(t_2-1)}{f(t_2)}+1\\
    &>\ r^2+1,
\end{split}
\end{equation} which leads to a contradiction as $r\geq 2\implies r^2> 3r/2$. Thus we must have $t_1=t_2+1$.
\\ \\ Now we know that
\begin{equation}
\begin{split}
f(n+2N+2)\ &<\ f_{n+2N+1}+f_{n+2N+2}\\
&=\ \frac{rsf_{t_1}}{g_{2N+1}} \ <\ f_{n+2N+3},
\end{split}
\end{equation} and therefore $c=rs / g_{2N+1}$ cannot possibly equal $1$. Lastly, we note that \begin{equation}
    \begin{split}
    \gcd(T_{m+1},T_m)\ &=\ \gcd(rT_m+T_{m-1},T_{m})\\
    &=\ \gcd(T_m,T_{m-1}).
    \end{split}
\end{equation}
Applying induction proves that this $\gcd$ is 1. The same argument shows that $\gcd(f_{m+1},f_m)=1$, but this contradicts the following statements:
\begin{align}
  sf(t_2+1)\ &=\ \frac{g(2N+1)}{r}T_{n+2N+2},\\
  sf(t_2)\ &=\ \frac{g(2N+1)}{r}T_{n+2N+1} \\
  \text{ and } c >1,
\end{align}
which completes our proof.
\end{proof}

\emph{Note that in the above proof the result does not hold for $r=1$, which is the Fibonacci sequence. We will prove that the sum of $4N+2$ consecutive Fibonacci numbers is a fixed integer multiple of another Fibonacci number (see Theorem \ref{thm:fibonacci-(4N+2)}).}

\subsection{Sums of Odd Numbers of Consecutive Terms}\label{OddPell}

\begin{thm}\label{thm:pell-odd}
 Let $P(n)$ denote the $n$\textsuperscript{${\rm th}$} Pell number. Fix any integer $N > 0$. There is no integer $C(N)$ such that for every $n$ there exists an integer index $j(n;N)$ such that the following equation holds:
 $$\sum_{i=0}^{2N} P(n+i)\ =\ C(N)P(j(n;N)).$$
\end{thm}

\begin{proof}
Suppose that we can write the sum of any $N$ consecutive Pell numbers as $C(N)$ times a Pell number for some positive integer $C(N)$ where $N$ is odd. Consider the Pell sequence modulo $C(N)$. By using the Pigeonhole Principle and the fact that two consecutive Pell numbers uniquely determine the terms before and after them we see that $\{ P_{n,C(N)}\}_{n\geq 0}:=\{P(n)\ \mbox{mod }C(N)\}_{n\geq 0}$ is periodic. The period is called the \emph{Pisano Period} and is denoted by $\pi(C(N))$.
\newline

Notice that $\pi(1)=1$ and $\pi(2)=2$. Now consider $C(N) > 2$. Since $\{ P_{n,C(N)}\}_{n\geq 0}$ is not a constant sequence, therefore $\pi(C(N)) \geq 2$, which implies that $$P_{c+\pi (C(N)),C(N)}=P_{c,C(N)}.$$
\\
Now since
\begin{align}
\begin{pmatrix} 2 &1 \\1& 0 \end{pmatrix}^{\pi(C(N))+1}\ =\ \begin{pmatrix} P(\pi(C(N))+2) &P(\pi(C(N))+1)\\P(\pi(C(N))+1)& P(\pi(C(N)))\end{pmatrix},
\end{align}
this implies that, with $I_2$ the $2\times 2$ identity matrix,
\begin{align}
\begin{pmatrix} 2 &1\\1& 0 \end{pmatrix}^{\pi (C(N))}\ =\ I_2 \in  GL_{2}\left( \mathbb{Z}/ \, C(N) \, \mathbb{Z}\right).
\end{align}
Taking the determinant, we get $(-1)^{\pi (C(N))}=1,$ which implies that $\pi(C(N))$ is even for all $C(N) > 2$.
\newline

Now let $N$ be an odd number, to emphasize this we change notation and write it as $2N+1$. Suppose the sum of any $2N+1$ consecutive Pell numbers is $C(N)$ times another Pell Number. Then
\begin{equation*}
    \sum\limits_{i=0}^{2N+1-1}P(n+i)\ \equiv\ 0 \pmod{C(N)},\mbox{ for } n\geq 0.
\end{equation*}
Replacing $n$ by $n+1$ we get
\begin{equation*}
    P(n + 2N+1)\ \equiv\ P(n) \pmod{C(N)}\; \mbox{ for all } n\geq 0,
\end{equation*} which implies that $\pi (C(N))\mid N$. However, since $\pi (C(N))$ is even when $C(N) \geq 2$, this implies that $C(N) = 1$. Thus the sum of any $2N+1$ consecutive Pell numbers must be equal to a Pell number. In other words,
\begin{equation*}
     \sum\limits_{i=0}^{2N+1-1}P(n+i)\ =\ P(j(n;N))
\end{equation*} for some integer $j(n;N)$ and $n\geq 0$. Notice that when $N = 0$, we obtain $j(n;N)=n$. Now suppose $2N+1$ is an odd integer greater than $1$. Then, we have
\begin{equation*}
    \sum\limits_{i=0}^{2N+1 - 1}P(n+i)\ >\ P(n+2N+1-1).
\end{equation*}
However,
\begin{equation*}
    \sum\limits_{i=0}^{2N+1-1}P(n+i) \ <\  \sum\limits_{i=0}^{n+2N+1-1}P(i)\ <\ P(n+2N+1)
\end{equation*} where the last inequality can be proven by using induction based on the value of $n+2N+1$. We conclude that $$P(n+2N+1-1) < \sum\limits_{i=0}^{2N+1-1}P(n+i) < P(n+2N+1),$$ and hence the sum of $2N+1$ consecutive Pell numbers is a fixed integer multiple of a Pell Number if and only if $N=0$.
\end{proof}

\emph{We use a proof of a similar flavor in Section \ref{gen_fib_seq}}.

\section{Fibonacci Numbers}

We now prove similar results for the Fibonacci numbers. In particular, we show that the sum of $N$ consecutive Fibonacci numbers is equal to a fixed constant multiple of a Fibonacci number if and only if $N\equiv 2\pmod{4}$, $N=3$, or $N=1$.

\subsection{Sum of $4N+2$ Consecutive Terms}

\begin{thm}\label{thm:fibonacci-(4N+2)}
Let $F(n)$ denote the $n$\textsuperscript{${\rm th}$} Fibonacci number, and $L(n)$ denote the $n$\textsuperscript{${\rm th}$} Lucas number. Fix any $N > 0$. The following equation
$$ \sum\limits_{i = 0}^{4N+1} F(n+i) \ =\ L(2N+1) F(n + 2N + 2),$$
holds for all $n$.
\end{thm}

\begin{proof}
A straightforward induction yields $$F(n+k)+(-1)^{k}F(n-k)\ =\ L(k)F(n),$$ and therefore we get
\begin{equation}
\begin{split}
\sum\limits_{i=0}^{4N+1}F(n+i)\ &=\ F(n+4N+3)-F(n+1)\\
&=\ L(2N+1)F(n+2N+2),\\
\end{split}
\end{equation}
which completes the proof.
\end{proof}

\subsection{Sum of $4N$ Consecutive Terms}
\begin{thm}\label{thm:fibonacci-(4N)}
Let $F(n)$ denote the $n$\textsuperscript{${\rm th}$} Fibonacci number. Fix any integer $N > 0$. There is no integer $C(N)$ such that for every $n$ there exists an integer index $j(n;N)$ such that the following equation holds:
 $$\sum_{i=0}^{4N-1} F(n+i)\ =\ C(N)F(j(n;N)).$$


\end{thm}

\begin{proof}
From Lemma \ref{lem:fib-(4N)} we have
\begin{equation*}
    \sum\limits_{i=0}^{4N-1}F(n+i)\ =\ F(2N)L(n+2N+1)\ =\ F(2N)(F(n+2N)+F(n+2N+2)).
\end{equation*} Now setting $T_m =F(m)+F(m+2)$ and repeating the proof of the $4N+2$ case for Pell numbers gives us the desired result.
\end{proof}

\subsection{Sums of Odd Numbers of Consecutive Terms}

We note that any Fibonacci number is one times itself and the sum of any three consecutive Fibonacci numbers is two times the third term. We prove that these are the only solutions for odd cases with the following theorem.

\begin{thm}
\label{thm:evenF}
 Let $F(n)$ denote the $n$\textsuperscript{${\rm th}$} Fibonacci number.  Fix any integer $N \geq 2$. There is no integer $C(N)$ such that for every $n$ there exists an integer index $j(n;N)$ such that the following equation holds:
$$\sum_{i=0}^{2N} F(n+i) \ = \ C(N) \cdot F(j(n;N)).$$
\end{thm}

\begin{proof}
    The proof of Theorem \ref{linear relations}, specifically, \eqref{integer_constant} tells us that if the sum of $N$ consecutive Fibonacci numbers is $C(N)$-times another Fibonacci number, then
   $$
    b\ =\ \dfrac{(\varphi ^N-1)}{C(N)(\varphi-1)}\ =\ \dfrac{\sum\limits_{i=0}^{N-1}\varphi ^{i}}{C(N)} \ =\ \varphi^{\gamma}\quad  \text{for some $\gamma \in \mathbb N$.}
 $$
    Using an argument of a similar flavor to that in Section \ref{OddPell}, we deduce that $C(N)$ must either be $1$ or $2$. Now, Lemma \ref{lem:sum-fib} tells us that
    \begin{align}
    b\ &=\ \frac{\varphi\sum\limits_{i=1}^{N-1}F(i) +\sum\limits_{i=1}^{N-2}F(i)+1}{C(N)} \ = \  \frac{(F(N+1)-1)\varphi+F(N)}{C(N)},
    \end{align}
     where $C(N)$ is either 1 or 2. Since for $n\geq 3$, $b\geq \alpha$, we let $b=\alpha ^{m}$ where $m\geq 1$. Thus, we get $$\frac{(F(N+1)-1)\varphi+F(N)}{C(N)}\ =\ F(m)\varphi +F(m-1),$$ which implies that $$F(m-1) \ = \ \frac{F(N)}{C(N)} \quad \text{ and } \quad F(m) \ = \ \frac{F(N+1)-1}{C(N)}.$$

We now consider the case $C(N) = 2$. Note that  $C(N) = 1$ is trivially not possible, the proof is left as an exercise to the reader.\footnote{Hint: if $C(N)=1$, we see that if $m\neq 3$ then $m=N+1$ leads to a contradiction, therefore $m=3$ and thus $N=1$.}

   \ \\
\textbf{Case: $C(N) = 2$.}
    \\  \\ If $C(N)=2$ then Carmichael's Theorem \cite{rdc} tells us that for $n>13$, $F(n)$ has a prime factor not present in the previous Fibonacci numbers. Therefore, we only need to check the cases where $n \leq 13$. Checking for the smaller cases we realize that $N=3$ is the only case where $F(N)/2$ is another Fibonacci number.
\end{proof}

\section{A result about general Lucas sequences}

Theorem \ref{thm:sum-4N-pell} shows that
\begin{equation}
\label{eq:Pell}
\sum_{i=0}^{4N-1} P(n+i)\ =\ 2P(2N) P(n+2N).
\end{equation}
This can be interpreted as saying that for all positive integers $M\equiv -1\pmod {4}$, and all positive integers $n$, there is a positive integer $m$ such that
$$
P(m)\mid \sum_{i=0}^M P(n+i)\qquad {\text{\rm and}}\qquad \frac{1}{P(m)} \sum_{i=0}^M P(n+i)\ =\ O(1).
$$
Indeed, to see this write $N=(M+1)/4$, and take $m=n+2N$. Then the amount $O(1)$ above is in fact the constant $2P(2N)\in {\mathbb Z}$. A similar formula holds for Fibonacci numbers by Theorem \ref{thm:fibonacci-(4N+2)}. Namely,
\begin{equation}
\label{eq:Fibonacci}
\sum_{i=0}^{4N+1} F(n+i) \ = \ L(2N+1)F(n+2N+2).
\end{equation}
This again implies that for all $M\equiv 1\pmod 4$ and all positive integers $n$ there is a positive integer
$m$ such that
$$
F(m)\mid \sum_{i=0}^M F(n+i)\qquad {\text{\rm and}}\qquad \frac{1}{F(m)} \sum_{i=0}^M F(n+i)\ =\ O(1).
$$
Indeed, for this just write $M=4N+1$ and take $m=n+2N+2$.
Other results  obtained so far show that such formulas do not exist for even $M$'s (see Theorems \ref{thm:pell-(4N+2)} and \ref{thm:evenF}).

We take a Lucas sequence $\{U(n)\}_{n\ge 0}$ of recurrence $U(n+2) \ = \ rU(n+1)+sU(n)$, where $r,~s$ are nonzero coprime integers. We assume that the characteristic equation $X^2-rX-s=0$ has two distinct roots $\alpha,~\beta$ such that $\alpha/\beta$ is not a root of $1$. In particular,
$\Delta \ := \ r^2+4s\ne 0$. Then
$$
U(n) \ = \ a\alpha^n+b\beta^n\qquad {\text{\rm holds~for~all}}\qquad n\ge 0.
$$
The above formula is called the Binet formula and it applies to any nondegenerate linearly recurrent sequence having the characteristic equation $X^2-rX-s=0$, and the coefficients $a,~b$ above can be calculated in terms of $U(0),~U(1)$.
Since we restrict our investigation to Lucas sequences, for us $(U(0),U(1)) =  (0,1),~(2,r)$ depending on whether we look at Lucas sequences of first or second kind, respectively. So, our Binet formula has
$$
(a,b)\ :=\ \left(\frac{1}{\alpha-\beta},\frac{-1}{\alpha-\beta}\right),\qquad {\text{\rm or}}\qquad (a,b) \ := \ (1,1).
$$
With this formalism we state the following theorem.
\begin{thm}
\label{thm:Lucas_Sequence_General}
Assume that there is an infinite sequence of positive integers $\{N_j\}_{j\ge 1}$ such that for each $N:=N_j$, there is a constant $K:=K_j$ and infinitely many positive integers $n$ such that
for each one of them there exists a positive integer $m$ with
$$
U(m)\mid \sum_{i=0}^N U(n+i)\quad {\text{\rm and}}\quad \left|\frac{1}{U(m)}\sum_{i=0}^N U(n+i)\right|\ \le\ K.
$$
Then either $s=-1$ or all three conditions $s=1$, $r\in \{\pm 1,\pm 2\}$ and $N$ is odd hold. Conversely, if either $s=-1$ or both $s=1$ and $r\in \{\pm 1,\pm 2\}$, then the above sequences $\{N_j\}_{j\ge 1}$ and constants $K_j$ exist.
\end{thm}

\begin{proof}
    For a proof, see Appendix \ref{subsec:Gen-Lucas}.
\end{proof}

\section{Generalized Pell and Fibonacci Numbers}

We adapt our previous results to a generalization of the Pell numbers that satisfies a   $(k+1)$\textsuperscript{{\rm st}} order recursion, where $k\in \bbn$. We also conjecture that for $k>1$, the sum of $N$ consecutive \emph{generalized Pell numbers} is a fixed integer multiple of another term of the sequence if and only if $N=2k+2$. Finally, we prove similar properties for a generalization of the Fibonacci numbers.

\subsection{Definition}
In \cite{ec} the authors consider the following generalization of the Pell numbers (we slightly modify their notation as we start our indexing at $n=0$).

\begin{dfn}
\emph{Generalized Pell $(k,i)$-numbers} are the solutions to the following recursion with given initial conditions:
\begin{align}
\label{rec:gen-pell}
&P_{k}^{i}(n)\ =\ 2P_{k}^{i}(n-1)+P_{k}^{i}(n-k-1)\\
\nonumber
&\text{with}\quad P_{k}^{i}(0)\ =\ P_{k}^{i}(1)\ = \ \cdots\  =\ P_{k}^{i}(i)\ =\ 0\\
\nonumber
&\text{and}\quad P_{k}^{i}(i+1)\ =\ P_{k}^{i}(i+2)\ =\ \cdots\ =\ P_{k}^{i}(k)\ =\ 1\\
\nonumber
&\text{where}\quad 0\ \leq\ i\ \leq\ k-1.\qquad (k\in\mathbb{N}).
\end{align}
\end{dfn}
%

\subsection{Sum of $2k+2$ Consecutive Terms}
Applying the following formula from \cite[\S4, Theorem 19]{ec}, we get
\begin{equation}\label{sumgenpell}
\sum\limits_{i=0}^{n}P_{k}^{k-1}(i)\ =\ \frac{1}{2}\left(-1+\sum\limits_{i=0}^{k}P_{k}^{k-1}(n-i+1)\right)
\end{equation}
where $n\geq k-1$. We prove a result similar to Theorem \ref{thm:sum-4N-pell} for the generalized Pell sequence.

\begin{thm} For $n \ge k$ we have
\begin{equation}\label{thm:gen-pell-1}
\sum\limits_{i=0}^{2k+1}P_{k}^{k-1}(n+i)\ =\
4P_{k}^{k-1}(n+2k).
\end{equation}
\end{thm}
\begin{proof}
For an algebraic proof, see Appendix \ref{subsec:gen-pell-1}. For a tiling argument see Theorem \ref{thm:tilings}.
\end{proof}
Now, setting $k=1$ in Theorem \ref{thm:gen-pell-1}, we obtain the following corollary.

\begin{cor}\label{cor:sum-4-Pell}
The sum of any four consecutive Pell numbers is four times the third of the Pell numbers:
\begin{equation}\label{eqn:sum-4-pell}
\sum\limits_{i=0}^{3}P(n+i)\ =\ 4 P(n+2).
\end{equation}
\end{cor}

We also obtain the following result.
\begin{thm} For $0 \le i \le k-1$ we have
\begin{equation} \sum\limits_{j=0}^{2k+1}P_{k}^{i}(n+j) \ =\ 4P_{k}^{i}(n+2k). \end{equation}
\end{thm}

\begin{proof}
    In \cite[\S2, Corollary 2]{ec} the authors prove for $n > k$ that
    \begin{equation}
    P_{k}^{k-1-j}(n)\ =\ P_{k}^{k-1}(n)+\sum\limits_{i=0}^{j-1}P_{k}^{k-1}(n-k+i).
    \end{equation}
This along with Theorem \ref{thm:gen-pell-1} gives the result.
\end{proof}

\subsubsection{\textbf{Conjecture}}
The partial sum formula of Pell numbers, along with the identity
\begin{equation}
    P(n+k)+(-1)^kP(n-k)\ =\ Q(k)P(n),\qquad k\in\mathbb{N}\cup \{ 0\}
\end{equation}
proven in Lemma \ref{lem:pell-(4N+2)}, can be used to give an alternate proof of Theorem \ref{thm:sum-4N-pell}. Although we have a similar partial sum formula for $P_{k}^{k-1}$, there is no obvious way to extend this partial sum to a general property of adding consecutive generalized Pell numbers to get a multiple of another generalized Pell number for arbitrary $k>1$. For $k>1$, we haven't been able to find $N\neq 2k+2$ such that the sum of N consecutive generalized Pell-$(k,i)$ numbers is an integer multiple of another generalized Pell-$(k,i)$, suggesting the following conjecture.

\begin{conj}
Fix any integer $N > 0$. There exists an integer $C(N)$ such that for every $n$ there exists an integer index $j(n;N)$ such that 
 $$\sum_{i=0}^{N} P_{k}^{i} (n+i) \ = \ C(N) \cdot P_{k}^{i} \, (j(n;N))$$ holds if and only if $N = 2k+2$.

\end{conj}

\subsection{Sum of Odd Number of Consecutive Terms}\label{generalized:pell}

The same argument given for the classical Pell numbers \eqref{thm:pell-odd} can be generalized for the sequence $\{ P_{k}^{k-1}(n)\}_{n\geq 0} $ where $k$ is an odd natural number. This is because from \cite[\S2,Theorem 2]{ec} we have the following equality:
\begin{eqnarray*}
& &
\begin{pmatrix}
2 & 0 & \cdots & 0 & 1 \\
1 & 0 & \cdots & 0 & 0 \\
0 & 1 & \cdots & 0 & 0 \\
\vdots & \vdots & \ddots & \vdots & \vdots \\
0 & 0 & \cdots & 1 & 0
\end{pmatrix}^{n}
\ = \ \nonumber\\
& & \begin{pmatrix}
P_{k}^{k-1}(n+p+1) & P_{k}^{k-1}(n+1) & \cdots & P_{k}^{k-1}(n+p-1) & P_{k}^{k-1}(n+p) \\
P_{k}^{k-1}(n+p) & P_{k}^{k-1}(n) & \cdots & P_{k}^{k-1}(n+p-2) & P_{k}^{k-1}(n+p-1) \\
P_{k}^{k-1}(n+p-1) & P_{k}^{k-1}(n-1) & \cdots & P_{k}^{k-1}(n+p-3) & P_{k}^{k-1}(n+p-2) \\
\vdots & \vdots & \ddots & \vdots & \vdots \\
P_{k}^{k-1}(n+1) & P_{k}^{k-1}(n-p+1) & \cdots & P_{k}^{k-1}(n-1) & P_{k}^{k-1}(n)
\end{pmatrix}.
\end{eqnarray*}
\\
Since
$$\det\begin{pmatrix} 2&0&\dots&0&1\\1&0&\dots &0&0\\0& 1&\dots&0&0\\ \vdots&\vdots &\dots &\vdots &\vdots \\0 &0&\dots& 1&0 \end{pmatrix}\ =\ (-1)^{k}\ =\ -1,$$
we obtain the following theorem.

\begin{thm}
Fix any integer $N > 0$. There is no integer $C(N)$ such that for every $n$ there exists an integer index $j(n;N)$ such that the following equation holds:
 $$\sum_{i=0}^{2N} P_{k}^{k-1} (n+i) \ =\ C(N) \cdot P_{k}^{k-1} \, (j(n;N)).$$
\end{thm}

\begin{proof}
    The proof follows from the discussion above and from the proof of the odd case for the classical Pell numbers.
\end{proof}

We now prove the following stronger theorem for the generalized Pell sequence.

\begin{thm}
Fix any odd integer $N > 0$, and even integer $k \geq 0$. Suppose that there exists an integer $C(N)$ such that for every $n$ there exists an integer index $\, j(n;N;k)\,$ such that the following equation holds:
\begin{equation*}
    \sum\limits_{i=0}^{N-1}P_{k}^{k-1}(n+i)\ =\ C(N) P_{k}^{k-1}(j(n;N;k)).
\end{equation*}
Then
\begin{enumerate}[label=\roman*)]
    \item $2\nmid C(N)$ and
    \item $N>2k+2$.
\end{enumerate}
\end{thm}

\begin{proof}
\begin{enumerate}[label=\roman*)]
    \item By induction we know that $a_n:\ =\ P_{k}^{k-1}(n)\pmod{2}$ is of the form
\begin{equation*}
    a_n\ =\ \left\{\begin{array}{ll}1,\mbox{ if }(k+1)\mid (n+1)\\ 0,\mbox{ otherwise.}\end{array}\right.
\end{equation*}

Let $N=q(2k+2)+r$. Since $N$ is odd, we have $1\leq r\leq 2k+1$. Now take any $n$ such that $n\geq k+1$ and $(k+1)\mid n$. We now prove that $2 \nmid C(N)$. Define $$\begin{cases}
S_{n,N}\ :=\ \sum\limits_{i=-r+1}^{q(2k+2)}P_{k}^{k-1}(n+i) & \text{ if $r \leq k+1$}, \\ \\
 S_{n,N}\ :=\ \sum\limits_{i=0}^{N-1}P_{k}^{k-1}(n+i) & \text{ if $ r > k+1$}.
\end{cases}$$ Using the explicit form of $a_n$ we conclude that $S_{n,N}$ is odd in both cases, and therefore $C(N)$ must be odd.

\item By the same argument, we know that $\pi (p)\mid N$ where $p$ is any prime dividing $C(N)$. Similarly, using the previous argument, we also know that $p>3$. Now since
\begin{equation*}
    P_{k}^{k-1}(0)\ =\ P_{k}^{k-1}(1)\ =\ \cdots \ =\ P_{k}^{k-1}(k-1)\ =\ 0,
\end{equation*}
and $P_{k}^{k-1}(k)=1$, we must have
\begin{equation*}
    P_{k}^{k-1}(\pi(p)) \ \equiv\ P_{k}^{k-1}(\pi(p)+1)\ \equiv\ \cdots\ \equiv\ P_{k}^{k-1}(k-1) \ \equiv\ 0 \pmod{p}.
\end{equation*}
However, we know that
\begin{equation*}
P_{k}^{k-1}(k+i)\ =\ 2^{i}\mbox{ for } 1\leq i\leq k,
\end{equation*} and since $k\geq 2$, we have
\begin{align*}
    &P_{k}^{k-1}(2k+1)\ =\ 2^{k+1}+1\\
    &P_{k}^{k-1}(2k+2)\ =\ 2^{k+2}+4\\
    &P_{k}^{k-1}(2k+3)\ =\ 2^{k+3}+12.
\end{align*}
Now as $p>2$, we know that $\pi(p)>2k$. We notice that if $\pi(p)=2k+1$, then $p\mid 2^{k+1}+1$ and $p\mid 2^{k+2}+4$ which implies $p\mid 2$: a contradiction. Similarly, if $p=2k+2$, then $p\mid 2^{k+2}+4 $ and $p\mid 2^{k+3}+12$ which implies $p\mid 4$ which is also a contradiction.
Therefore, $\pi(p)>2k+2$ which implies $N>2k+2$.
\end{enumerate}
\end{proof}

\subsection{Tilings and Generalized Pell Sequence}
In \cite{asj} the authors proved certain properties related to the Pell numbers using tilings of an $n\times 1$ board. We generalized some of the properties for $P_{k}^{k-1}(n)$.
Let us first define a sequence $(p_{k,n})_{n\geq 0}$ such that
\begin{equation}
p_{k,n}\ :=\ P_{k}^{k-1}(n+k).
\end{equation}
It is not difficult to see that $p_{k,n}$ counts the number of tilings of an $n\times 1$ board using black $1\times 1$ squares, white $1\times 1$ squares and grey $(k+1)\times 1$ polyominoes.

\begin{thm}\label{thm:tilings} We have
\begin{equation}
p_{k,(k+1)n+r+1}\ =\ \left\{\begin{array}{ll}
\ds 2\sum\limits_{m=0}^{n}p_{k,m(k+1)+r},& 0\leq r\ <\ k\\
\ds 2\sum\limits_{m=0}^{n}p_{k,m(k+1)+r}+1,& r\ =\  k.\end{array}\right.
\end{equation}
\end{thm}

\begin{proof}
Firstly, assume that $r<k$. Now, consider the tiling of a $[(k+1)n+r+1]\times 1$ board, with the cells on the board numbered from left to right $1$ to $(k+1)n+r+1$. Let $t$ be the location of the last $1\times 1$ cell in the tiling. Black or white squares cannot cover the cells to the right of $t$, so they must be covered by $(k+1)\times 1$ polyominoes. Therefore, $t$ is of the form $(k+1)m+r+1$. In this case, the number of tilings of the board is $2p_{k,mk}$ (accounting for the fact that cell $t$ can be covered by either black or white $1\times 1$ squares), proving the identity.
\\ \\
\noindent Now, let us assume that $r=k$. We can still cover the board with black and white squares as well as grey polyominoes as we discussed in the previous case, yielding $2\sum\limits_{m=0}^{n}p_{k,m(k+1)+r}$ tilings of the board. However, since the length of the board is now $(k+1)(n+1)$, it is possible the board can be covered without black and white squares altogether. We add this new case to the total number of tilings, proving the second identity.
\end{proof}

Note that \eqref{sumgenpell} also follows from this result. An alternate proof using matrices is given in \cite[\S4, Theorem 19]{ec}, which can be generalized further. \\ \\
\begin{dfn}
Define the sequence $\left\{ f_k(n)\right\}$ as follows:
\begin{equation*}
\begin{aligned}
    f_{k}(n)\ := \ af_{k}(n-1)+bf_{k}(n-k-1)\hspace{20pt} a,b\in\bbn\\
    \mbox{with } f_{k}(0)\ =\ f_{k}(1)\ =\ \cdots \ =\ f_{k}(k-1)\ =\ 0\;\; \ \text{and } \ f_k(k)=1 \quad \mbox{ where }k\in\mathbb{N}.
\end{aligned}
\end{equation*}
\end{dfn}
\begin{dfn}
Define the sequence $\left\{ p_{k,n}\right\}$ as follows:
$$p_{k,n}\ :=\ f_{k}(n+k)\mbox{ for }n\in\mathbb{N}\cup \{ 0\}.$$
\end{dfn}
\begin{thm}
We have
\begin{equation}\label{gseqsum}
p_{k,(k+1)n+r+1}\ =\ \left\{\begin{array}{ll}a\sum\limits_{m=0}^{n}b^{n-m}p_{k,m(k+1)+r}, & 0\leq r\ <\ k \\ \\ a\sum\limits_{m=0}^{n}b^{n-m} p_{k,m(k+1)+r}+1, & r= k.\end{array}\right.
\end{equation}
\end{thm}

\begin{proof}
    Analogous to the proof of Theorem \ref{thm:tilings}.
\end{proof}

\subsection{Generalized Fibonacci sequence} \label{gen_fib_seq}
Define the order-$k$ generalized Fibonacci sequence by
\begin{align}
    &f_k(n)\ :=\ \sum\limits_{i=1}^{k}f_k(n-i)\\
    \nonumber
    &\mbox{ with }\ f_k(1)\ =\ f_k(2)\ =\ \cdots \ =\ f_k(k-1)\ =\ 0\mbox{ and }f_k(k)\ =\ 1.
\end{align}
Its generating matrix (see \cite{ed}) is given by
 \begin{equation}
     \begin{pmatrix} 1&1&\dots&1&1\\1&0&\dots &0&0\\0& 1&\dots&0&0\\ \vdots&\vdots &\dots &\vdots &\vdots \\0 &0&\dots& 1&0 \end{pmatrix}\\.
 \end{equation}

A similar argument to the generalized Pell case in Section \ref{generalized:pell} tells that the Pisano Period for $f_k(n)$ is even modulo $n$ whenever $n>2$ and $k$ is even, and yields the following Theorem.

\begin{thm}\label{thm:generalized_fibonacci_2n+1}
 Let $F_k(n)$ denote the $n$\textsuperscript{${\rm th}$} order-$k$ Fibonacci number where $k$ is even. Fix any $N > 0$. There is no integer $C(N)$ such that for every $n$ there exists an integer index $\, j(n;N)\,$ such that the following equation holds: $$\sum_{i=0}^{2N} F_k(n+i)\ =\ C(N) \cdot F_k(\, j(n;N)\,).$$
\end{thm}
\begin{proof}The proof of the theorem is analogous to Theorem \ref{thm:pell-odd}. For further details, see Appendix
\ref{subsec:proof_generalized_fibonacci_2n+1}.
\end{proof}

\newpage
\appendix

\section{Proofs}

\subsection{Proof of Lemma \ref{lem:pell-(4N+2)}}\label{subsec:proof_lem:pell-(4N+2)}
\begin{proof}
We proceed by induction on $k$, noting that the two base cases are $k=0$ and $k=1$. When $k=0$, we have
\begin{equation}\label{eqn:pell-(4N+2)-base-case-1}
P(n+0) + (-1)^0 P(n-0) \ = \  2P(n) = Q(0)P(n).
\end{equation}
When $k=1$, we have
\begin{equation}\label{eqn:pell-(4N+2)-base-case-2}
P(n+1) + (-1)^1 P(n-1) \ = \  [2P(n) + P(n-1)] - P(n-1) \ = \  2P(n) \ = \ Q(1)P(n).
\end{equation}
Now, we assume that
\begin{equation}\label{eqn:pell-(4N+2)-ind-1}
P(n+k-1)+(-1)^{k-1} P(n-k+1)\ =\ Q(k-1) P(n)
\end{equation}
and
\begin{equation}\label{eqn:pell-(4N+2)-ind-2}
P(n+k-2)+(-1)^{k-2} P(n-k+2)\ =\ Q(k-2) P(n).
\end{equation}
Using the recurrence relation \eqref{rec:pell}, we have
\[P(n+k)\ =\ 2 P(n+k-1) + P(n+k-2).\]
Rearranging \eqref{eqn:pell-(4N+2)-ind-2}, we get
\[P(n-k+2)\ =\ 2P(n-k+1)+P(n-k) \implies
P(n-k)\ =\ P(n-k+2)-2P(n-k+1).\]
Thus,
\begin{align}\label{eqn:pell-(4N+2)-1}
&P(n+k)+(-1)^kP(n-k) \nonumber\\
&= \ 2P(n+k-1)+P(n+k-2)+(-1)^k(P(n-k+2)-2P(n-k+1)).
\end{align}
Rearranging the right-hand side of \eqref{eqn:pell-(4N+2)-1} yields
\begin{equation}\label{eqn:pell-(4N+2)-2}
2(P(n+k-1)+(-1)^{k-1}P(n-k+1))+(P(n+k-2)+(-1)^{k-2}P(n-k+2)).
\end{equation}
We apply the inductive hypotheses \eqref{eqn:pell-(4N+2)-ind-1} and \eqref{eqn:pell-(4N+2)-ind-2} along with Definition \ref{dfn:pell-lucas} to this expression to conclude
\begin{align*}
2Q(k-1)P(n)+Q(k-2)P(n)\
&=\ (2Q(k-1)+Q(k-2)) P(n)\\
&=\ Q(k)P(n).
\end{align*}\end{proof}

\subsection{Proof of Theorem \ref{lem:fib-(4N)}}
\label{subsec:proof_lem:fib-(4N)}
\begin{proof}
We use induction and the following well-known properties of the Fibonacci and Lucas Numbers \cite[\S5.3, \S5.8]{tk}:
\begin{align}
&(i)\quad F(n-1)F(n+1)-F(n)^2\ =\ (-1)^n.\label{Fib:Prop-1}\\
&(ii)\quad F(n+k)\ =\ F(n)F(k-1)+F(n+1)F(k).\label{Fib:Prop-2}\\
&(iii)\quad F(n-1)+F(n+1)\ =\ L(n).\label{Fib:Prop-3}\\
&(iv)\quad \sum\limits_{i=0}^{n-1}F(i)+1\ =\ F(n+1).\label{Fib:Prop-4} \\
&(v)\quad \sum\limits_{i=0}^{k} F(n+i) = F(n+k+2) - F(n+1). \label{Fib:Prop-5}
\end{align}

We now prove Lemma \ref{lem:fib-(4N)}. First, begin by noting that for $n=0$, we have
\begin{align*}
\sum\limits_{i=0}^{4N-1}F(i)\ &=\ F(4N+1)-1\qquad \mbox{ (Using \eqref{Fib:Prop-4})}\\
&=\ F(2N)F(2N+2)+F(2N-1)F(2N+1)-1\qquad \mbox{ (Using \eqref{Fib:Prop-2})}\\
&=\ F(2N)(F(2N+2)+F(2N))\qquad \mbox{(Using \eqref{Fib:Prop-1}}\\
 &=\ F(2N)L(2N+1).\qquad \mbox{(Using \eqref{Fib:Prop-3})}\\
\end{align*}

Now, by our induction hypothesis, $\sum\limits_{i=0}^{4N-1} F(m+i) \ =\ F(2N)L(m + 2N + 1)$ holds for all $m < n+1$. We now expand $\sum\limits_{i=0}^{4N-1} F(n+1+i)$ using the following manipulations:

\begin{align*}
\sum\limits_{i=0}^{4N-1}F(n+1+i)\ &=\ F(n+4N+2)-F(n+2) \mbox{ (Using \eqref{Fib:Prop-5})}\nonumber \\
&= \ F(n+4N+1)+F(n+4N)-(F(n+1)+F(n)) \nonumber\\ & \ \ \ \ \ \ \mbox{(Using \eqref{def:fibonacci})}\\
&=\ F(n+4N+1)-F(n+1)+(F(n+4N)-F(N))  \nonumber\\ & \ \ \ \ \ \ \mbox{(Rearranging terms)}\\
&=\ \sum\limits_{i=0}^{4N-1}F(n+i)+\sum\limits_{i=0}^{4N-1}F(n-1+i) \nonumber\\ & \ \ \ \ \ \  \mbox{ (Using \eqref{Fib:Prop-5})}\\
&= \ F(2N)(L(n+2N+1)+L(n+2N)  \nonumber\\ & \ \ \ \ \ \  \mbox{(Induction hypothesis)}\\
&=\ F(2N)L(n+2N+2) \qquad \mbox{(Using Definition \ref{dfn:Lucas})},
\end{align*}\end{proof}
which yield the desired result.

\subsection{Proof of Theorem \ref{thm:non_degenerate_general_4}}\label{subsec:proof_specialized_linear_relations}
\begin{proof}
    From the proof of Theorem \ref{linear relations} we have
       $ C(N) = \sum\limits_{i=0}^{3}\alpha^{i-k(N)}$ for $C(N)$ a positive integer, where $2 \leq k(N) \leq 4$. Therefore, the only possible values for $k(N)$ are $2, 3, 4$. We now do casework based on the value of $k(N)$.\\ \

   \textbf{Case 1: $k(N) =2$.} We have

$$
\begin{aligned}\label{eqn:sum-4-c:2}
    C(N) &\ = \  \frac{1}{\alpha^2} + \frac{1}{\alpha} + 1 + \alpha \\
         &\ = \  1 + \frac{2r^2+4-2r\sqrt{r^2+4}}{4} + \frac{\sqrt{r^2+4}-r}{2} + \frac{r+\sqrt{r^2+4}}{2} \\
         &\ = \  1 + \frac{r^2+2}{2} + \left(\frac{2-r}{2}\right)\sqrt{r^2+4}.
\end{aligned}
$$
Since $r$ is an integer and there is no Pythagorean triple with $2$ as one of the terms, therefore $\sqrt{r^2+4}$ is irrational. Thus for $C(N)$ to be an integer, we must have $\frac{2-r}{2} = 0$, therefore $r=2$. \\ \\ \textbf{Case 2: $k(N) =3$.} We have
 $$
   \begin{aligned}
   C(N) \ &=\ \frac{1}{\alpha^3} + \frac{1}{\alpha^2} + \frac{1}{\alpha} + 1\nonumber\\
        &=\ \frac{(r^2+4)\sqrt{r^2+4}-3r(r^2+4)+3r^2\sqrt{r^2+4}-r^3}{8} \nonumber\\ &+\frac{2r^2+4-2r\sqrt{r^2+4}}{4} + \frac{\sqrt{r^2+4}-r}{2}\nonumber\\
        &=\ \frac{-4r^3+4r^2-16r+8}{8} + \left( \frac{4r^2+4 -4r+4}{8}\right)\sqrt{r^2+4}\nonumber\\
        &=\ \frac{-r^3+r^2-4r+2}{2} + \left( \frac{r^2-r+2}{2}\right)\sqrt{r^2+4}.
\end{aligned}
 $$

    Since $C(N)$ is an integer, we must have $r^2-r+2=0$. Since this equation has no integer roots, no such $r$ exists.
\ \\  \\
\textbf{Case 3: $k(N) =4$.} We have
    \begin{align}
       C(N) \ &= \ \frac{1}{\alpha^4} + \frac{1}{\alpha^3} + \frac{1}{\alpha^2} + \frac{1}{\alpha}\\ \nonumber \
       &=\ \frac{r^4-r^3+5r^2-3r=4}{2}+\left(\frac{-r^3+r^2-3r+1}{2}\right)\sqrt{r^2+4}.
    \end{align}
    Since $C(N) $ is an integer, we must have $-r^3+r^2-3r+1=0$. As this has no integer roots, so no such $r$ exists.
\end{proof}

\subsection{Proof of Theorem \ref{thm:Lucas_Sequence_General}}\label{subsec:Gen-Lucas}
\renewcommand{\thesection}{A.4}

We label the roots in such a way that $|\alpha|\ge |\beta|$. In particular, $|\alpha|>1$. Suppose for now that $\beta\ne \pm 1$ and we shall return to this case later.
We take an $N$ and calculate
$$
\sum_{i=0}^N U(n+i) \ = \ \sum_{i=0}^N (a\alpha^{n+i}+b\beta^{n+i}) \ = \ a\left(\frac{\alpha^{N+1}-1}{\alpha-1}\right)\alpha^{n}+b\left(\frac{\beta^{N+1}-1}{\beta-1}\right)\beta^n.
$$
Let us give upper and lower bounds for the size of $\sum_{i=0}^N U(n+i)$. First, when is it zero?

\begin{lemma}
We have that
$$
\sum_{i=0}^N U(n+i) \ \ne \  0
$$
for $N>29$.
\end{lemma}

\begin{proof} Assume that $\sum_{i=0}^N U(n+i)=0.$ Then

$$
a(\alpha^{N+1}-1)\alpha^n \ = \ -b(\beta^{N+1}-1)\beta^n.
$$
Since $a/b \ = \ \pm 1$, we get that
$$
(\alpha^{N+1}-1)\alpha^n \ = \ \eta (\beta^{N+1}-1)\beta^n\quad {\text{\rm for~some}}\quad \eta\in \{\pm 1\}.
$$
This gives
$$
\alpha^{n+N+1}-\eta \beta^{n+N+1} \ = \ \alpha^n-\eta \beta^n.
$$
In both cases $\eta=1$ and $\eta=-1$, this shows that the $n+N+1$th member of the Lucas sequence of the first or second kind equals the $n$th member of the same sequence. But recall that members of Lucas sequences have primitive divisors, which are prime numbers $p$ which divide the term of index say $m$ but no term of index smaller than $m$, at least when $m>30$ by results of Bilu, Hanrot and Voutier \cite{BHV}. Thus, we get that $n+N+1\le 30$ in contradiction with $N>29$.
\end{proof}

Since now we know that $\sum_{i=0}^N U(n+i)$ is not zero, we then get
\begin{eqnarray}
\label{eq:lflogs}
\left|\sum_{i=0}^{N} U(n+i)\right| &\ =\ & |a| \left|\frac{\alpha^{N+1}-1}{\alpha-1}\right||\alpha|^n\left|1+\left(\frac{b(\beta^{N+1}-1)(\alpha-1)}{a(\alpha^{N+1}-1)(\beta-1)}\right)\left(\frac{\beta}{\alpha}\right)^n\right|\nonumber\\
& \gg_N  & |\alpha|^{n-C_1(N)\log n},
\end{eqnarray}
where in the left--hand side above we applied a lower bound for a linear form in logarithms of algebraic numbers\footnote{For functions $f(x)$ and $g(x)$, we write $f(x) \gg g(x)$ if there exists a constant $C > 0$ such that $f(x) \geq C \cdot g(x)$ for all sufficiently large $x$.}. Here, $C_1(N)$ is a constant depending both on $N$ and on the Lucas sequence $\{U(m)\}_{m\ge 0}$. Now we want that $U(m)\mid \sum_{i=0}^N U(n+i)$ and the co-factor to be bounded. We get
$$
\left|\alpha\right|^m \ \gg \  |U(m|) \ \gg \ |\sum_{i=0}^N U(n+i)| \ \gg_N  \ |\alpha|^{n-C_1(N)\log n},
$$
getting $m \ \ge \  n-C_2(N)\log n$. The constant $C_2(N)$ depends also on the bound on the co-factor
$$
\frac{1}{U(m)}\sum_{i=0}^N U(n+i).
$$
In case the roots $\alpha,~\beta$ of $\{U(m)\}_{m\ge 0}$ are real, Baker's method is not necessary since $|\alpha|>|\beta|$ and the above argument shows that in fact $m\ge n-O(1)$.
Furthermore,
$$
|\alpha|^{m-C_2\log m} \ \ll \  |U_m| \ \le\ \left|\sum_{i=0}^N U_{n+i}\right|\ \ll_N \  |\alpha|^n
$$
(for the left--hand inequality see Theorem 3.1 on page 64 in \cite{ST}), which in turn gives $m\le n+C_3(N)\log n$. This shows that $m=n+O(\log n)$. Let us record this as a lemma.

\begin{lemma}
\label{lem:m}
If $N>29$, and $U_m\mid \sum_{i=0}^N U_{n+i}$ is such that the co-factor remains bounded, then
$$
m \ = \ n+O(\log n).
$$
The constant implied by the $O$-symbol above depends on the sequence $\{U(m)\}_{m\ge 0}$, the bound on the co-factor and the number $N$.
\end{lemma}

Next, let $D:=U(m)$ and let us write down the congruences
$$
\sum_{i=0}^n U(n+i)\ \equiv \ 0\pmod D,\qquad U(m)\ \equiv \ 0\pmod D.
$$
This we rewrite as
$$
a\left(\frac{\alpha^{N+1}-1}{\alpha-1}\right)\alpha^n \ \equiv \  -b\left(\frac{\beta^{N+1}-1}{\beta-1}\right)\beta^n\pmod D,\quad a\alpha^m \ \equiv \  -b\beta^m \pmod D.
$$
Since $a=\pm b$ and $a,~b$ are reciprocals of algebraic integers, we can simplify by $a$ the above relations getting
$$
\left(\frac{\alpha^{N+1}-1}{\alpha-1}\right)\alpha^n \ \equiv \ \eta \left(\frac{\beta^{N+1}-1}{\beta-1}\right)\beta^n\pmod D,\quad \alpha^m \ \equiv \ \eta\beta^m\pmod D,
$$
where $\eta:=-b/a\in \{\pm 1\}$. Since $\alpha,~\beta$ are coprime as algebraic integers (that is the ideals generated by them in ${\mathcal O}_{\mathbb K}$, where ${\mathbb K}:={\mathbb Q}(\alpha)$ are coprime since $r$ and $s$ are coprime integers),
we get that $\beta$ is invertible modulo $D$. Thus, we get
$$
\left(\frac{\alpha^{N+1}-1}{\alpha-1}\right) \left(\alpha/\beta\right)^n \ \equiv  \ \eta\left(\frac{\beta^{N+1}-1}{\beta-1}\right)\pmod D,\qquad (\alpha/\beta)^m \ \equiv \ \eta\pmod D.
$$
Let $D_1:=\gcd(\alpha^{N+1}-1,D)$ as an ideal of ${\mathcal O}_{\mathbb K}$, and let $D_2:=D/D_1$. We then get
\begin{eqnarray}
\label{eq:1}
(\alpha/\beta)^{2n}  & \ \equiv \ &  \gamma\pmod {D_2},\quad \gamma \ := \ \left(\frac{(\beta^{N+1}-1)(\alpha-1)}{(\alpha^{N+1}-1)(\beta-1)}\right)^2,\nonumber\\
(\alpha/\beta)^{2m} & \ \equiv \ & 1\pmod {D_2}.
\end{eqnarray}
Clearly,
$$
|N_{{\mathbb K}/{\mathbb Q}}(D_1)| \ \le \ |N_{{\mathbb K}/{\mathbb Q}}(\alpha^{N+1}-1)| \ = \ O_N(1).
$$
 Recall the following lemma.
\begin{lemma}
If $\gamma$ and $\alpha/\beta$ are not multiplicatively dependent, then \eqref{eq:1} shows that
$$
N_{{\mathbb K}/{\mathbb Q}}(D_2) \ = \ \exp(O_N({\sqrt{ n}})).
$$
\end{lemma}
 \begin{proof}
 The proof is standard and can be found in many places. Put $X:=\max\{n,m\}$. Note that $X=O(n)$. By the pigeon hole principle there are integers $u,v$ not both zero with $|u|\le {\sqrt{X}},~|v|\le {\sqrt{X}}$ such that
 $|un+mv|=O({\sqrt{X}})$. Exponentiating in \eqref{eq:1} the first equation to $u$ and the second to $v$ and multiplying them we get
 $$
 (\alpha/\beta)^{2nu+2mv} \ \equiv \ \pm \gamma^{u}\pmod {D_2}.
 $$
 Thus, $D_2$ divides one of $(\alpha/\beta)^{|2nu+2mv|}\pm \gamma^{|u|}$ or $(\alpha/\beta)^{|2nu+2mv|}\gamma^{|u|}\pm 1$. None of these expressions is $0$ if $\gamma$ and $\alpha/\beta$ are multiplicatively independent (indeed, such expressions being $0$ and $(\alpha/\beta)$ and $\gamma$ being multiplicatively independent first forces $u=0$, then $2un+2vm=0$, so also $v=0$, which is not allowed), and their
 norms are rational numbers whose denominator and numerator are of size $\exp(O({\sqrt{X}}))=\exp(O({\sqrt{n}}))$.
 \end{proof}

In conclusion, if $\gamma$ and $\alpha/\beta$ are multiplicatively independent, then
\begin{eqnarray*}
|\alpha|^{n-C_3(N)\log n} & \ \le \  &  |U_m|\le |N_{{\mathbb K}/{\mathbb Q}}(D_1D_2)| \ \le  \ |N_{{\mathbb K}/{\mathbb Q}}(D_1)||N_{{\mathbb K}/{\mathbb Q}}(D_2)|\\
&  \ \le \ & \exp(O_N({\sqrt{n}})),
\end{eqnarray*}
and this is false for any fixed $N>29$ and large $n$. So, $\gamma$ and $\alpha/\beta$ are multiplicatively dependent.

\begin{lemma}
If $\gamma$ and $\alpha/\beta$ are multiplicatively dependent, then $s=\pm 1$. Furthermore, if $s=1$ then $N$ is odd.\end{lemma}

\begin{proof}
Assume that $\gamma$ and $\alpha/\beta$ are multiplicatively dependent. Then $\gamma^u=(\alpha/\beta)^v$ for some integers $u,~v$ not both $0$. If $u=0$, then $(\alpha/\beta)^v=1$ for some nonzero integer $v$ which makes
$\alpha/\beta$ a root of $1$ and this is not allowed. Thus, $u\ne 0$. We get
\begin{equation}
\label{eq:3}
\left(\frac{\alpha^{N+1}-1}{\beta^{N+1}-1}\right)^u \ = \ \left(\frac{\alpha}{\beta}\right)^v \left(\frac{\alpha-1}{\beta-1}\right)^u.
\end{equation}
Let $S$ be the set of $S$-units with respect to the finite set of valuations $\nu$ of ${\mathbb  K}$ such that $\nu$ is either infinite or one of
$$|\alpha-1|_{\nu},\quad |\beta-1|_{\nu},\quad |\alpha|_\nu,\quad |\beta|_{\nu}
$$
is not equal to $1$. Equation \eqref{eq:3} signals $(\alpha^{N+1}-1)/(\beta^{N+1}-1)=s_0$ for some $s_0\in S$. This gives
\begin{equation}
\label{eq:Sunit}
\alpha^{N+1}-s_0\beta^{N+1}+s_0 \ = \ 1,
\end{equation}
which is an ${\mathcal S}$-unit equations in three terms.
This has only finitely many solutions $N$ which are nondegenerate. The degenerate solutions correspond to the following.
\begin{itemize}
\item[(i)] $\alpha^{N+1}=1$. In this case also $\beta^{N+1}=1$, so $(\alpha/\beta)^{N+1}=1$, which is impossible.
\item[(ii)] $\alpha^{N+1}+s_0=0$. Then $s_0=-\alpha^{N+1}$ and $1=-s_0\beta^{N+1}=(\alpha\beta)^{N+1}=(-s)^{N+1}$. Thus, $s\in \{\pm 1\}$ and $N+1$ is even when $s=1$.
\item[(iii)] $\alpha^{N+1}-s_0\beta^{N+1}=0$. Then also $1=s_0=(\alpha/\beta)^{N+1}$ and this is not allowed.
\end{itemize}
Thus, the only possibilities are $s=-1$ or both $s=1$ and $N$ odd.

It remains to show that if $s=1$, only $r=\pm 1$ and $r=\pm 2$ work. Note that if the Lucas sequence of roots $(\alpha,\beta)$ satisfies the above properties, so will the ones of roots $(-\alpha,-\beta)$ as the ratio $\alpha/\beta$ does not change.
Thus, in case $s=1$ we assume that $\alpha>1$, so $r$ is positive.

We take another look at $\gamma$. When $s=-1$, we have $\beta=1/\alpha$ so
$$
\gamma=\left(\frac{(\beta^{N+1}-1)(\alpha-1)}{(\alpha^{N+1}-1)(\beta-1)}\right)^2 \ = \ \left(\frac{(1/\alpha^{N+1}-1)(\alpha-1)}{(\alpha^{N+1}-1)(1/\alpha-1)}\right)^2 \ = \ \alpha^{-2N},
$$
which of course is multiplicatively dependent over $\alpha/\beta=\alpha^2$.

Assume next that $s=1$. Since $N+1$ is even, then
$$
\gamma^2 \ =\ \left(\frac{(1/\alpha^{N+1}-1)(\alpha-1)}{(\alpha^{N+1}-1)(-1/\alpha-1)}\right)^2 \ =\ \alpha^{-2N} \left(\frac{\alpha-1}{\alpha+1}\right)^2.
$$
So, we want
$$
\gamma_1 \ := \ \frac{\alpha-1}{\alpha+1}
$$
to be multiplicatively dependent with $\alpha/\beta=-\alpha^2$. The above number is
$$
\frac{(r-2+{\sqrt{r^2+4}})/2}{(r+2+{\sqrt{r^2+4}})/2}.
$$
The numbers $(r\pm 2)+{\sqrt{r^2+4}})/2$ are algebraic integers of norms
$$((r\pm 2)^2-(r^2+4))/4 \ = \ \pm r.
$$
So, if $r$ is odd, they have odd norms. Their greatest common divisor (as ideal) divides $(\alpha+1)-(\alpha-1)=2$, but the norms of $\alpha\pm 1$
are coprime to $2$. Hence, $\alpha\pm 1$ are coprime. This shows that since their ratio is a unit, each one of them is a unit. Thus, $\alpha+1$ and $\alpha$ are both units and are larger than $1$. But letting $\delta$ be the fundamental unit in ${\mathbb K}={\mathbb Q}(\alpha)$, all units larger than $1$ are of the form
$$
1,~\delta,~\delta^2,\ldots,
$$
and $\delta\ge (1+{\sqrt{5}})/2$. Thus, since $\delta^3-\delta^2\ge \delta>1$, the only possibility is $\delta^2=\alpha+1$ and $\delta=\alpha$, giving $\delta^2=\delta+1$, so $\delta=\alpha=(1+{\sqrt{5}})/2$. This corresponds to $r=1$.

Assume next that $r$ is even and write $r=2r_0$. Then $\alpha\pm 1=r_0\pm 1+{\sqrt{r_0^2+1}}$. Suppose that $r_0$ is even. Then
$$
\frac{\alpha\pm 1}{2} \ = \ \frac{r_0\pm 1+{\sqrt{r_0^2+1}}}{2}
$$
are algebraic integers which are consecutive (their difference is $1$) and their ratio is a unit. So, each one of them is a unit. Thus, we arrive again at a quadratic field in which we have two units larger than $1$ which are
$(r_0\pm 1+{\sqrt{r_0^2+1}})/2$ whose difference is $1$. This gives
$$
\frac{r_0+1+{\sqrt{r_0^2+1}}}{2} \ = \ \frac{1+{\sqrt{5}}}{2},
$$
and this is impossible. Finally, suppose that $r_0$ is odd. Look at the numbers
$$
\frac{\alpha\pm 1}{\sqrt{2}} \ = \ \frac{r_0\pm 1+{\sqrt{r_0^2+1}}}{\sqrt{2}}.
$$
Their norms divide
$$
\left(\frac{r_0^2\pm 2r_0+1-(r_0^2+1)}{2}\right)^2 \ = \ r_0^2,
$$
so they have odd norms. This shows that $(\alpha+1)^2/2$ and $(\alpha-1)^2/2$ are coprime and their ratio is a unit so each one of them is a unit. In particular,
$$
\frac{\alpha+2+1/\alpha}{2}\qquad {\text{\rm and}}\qquad \frac{\alpha-2+1/\alpha}{2}
$$
are both units which differ by $2$. Since $r$ is even, we get that $\alpha\ge 1+{\sqrt{2}}$. The case $r=2$ gives us what we want. The case $r\ge 6$ gives $\alpha>6$, so $(\alpha-2+1/\alpha)/2>2$. Thus, putting again $\delta$ for the fundamental unit
in ${\mathbb Q}(\alpha)$ we get that $(\alpha+2+1/\alpha)/2$ and $(\alpha-2+1/\alpha)/2$ are among
$$
\delta,~\delta^2,~\delta^3,~\ldots
$$
and their difference is $2$. Since $\delta^3-\delta^2=\delta^2(\delta-1)\ge (1+{\sqrt{2}})^2(1+{\sqrt{2}}-1)>2$, we get that $\delta^2-\delta=2$, which has no quadratic solution $\delta$. This finishes the proof.

It remains to deal with the converses. Theorems  \ref{thm:sum-4N-pell}  and \ref{thm:fibonacci-(4N+2)} do it for the case $s=1$ and $r=1,2$ for the Lucas sequence of the first kinds (Fibonacci and Pell numbers). Similar formulas exist
with  Lucas numbers and Lucas-Pell numbers (so the Lucas sequences of the second kind which are companions of the Fibonacci and Pell numbers, respectively). Namely, with $\{L(n)\}_{n\ge 0}$ and
$\{Q(n)\}_{n\ge 0}$ given by $L(0)=Q(0)=2,~L(1)=1,~Q(1)=2$, we get
$$
\sum_{i=0}^{4N+1} L(n+i) \ = \ L(2N+1)L(2N+n+2)
$$
and
$$
\sum_{i=0}^{4N-1} Q(n+i) \ = \ 2P(2N) Q(2N+n).
$$
We also need to deal with $r=-1,-2$. That is, when $r=-1,~-2$, we get $U(n)=(-1)^nF(n)$ and $U(n)=(-1)^n P(n)$ for the Lucas sequences of the first kind. We get the identities
$$
\sum_{i=0}^{4N+1} (-1)^{n+i} F(n+i) \ = \ L(2N+1) ((-1)^{2N+n-1} F(2N+n-1))
$$
and
$$
\sum_{i=0}^{4N-1} (-1)^{n+i} P(n+i) \ = \ 2P(2N) ((-1)^{2N+n-1}P(2N+n-1)).
$$
These formulas are not unexpected since $(-1)^n F(n)=F(-n)$. In particular, the above two formulas follow right--away from  \eqref{eq:Fibonacci} and \eqref{eq:Pell}
if one allows the index $n$ to be negative in formulas \eqref{eq:Fibonacci} and \eqref{eq:Pell}. Similar formulas exist with the Lucas numbers of the second kind when $s=1$ and $r=-1,~-2$.
To see the case $s=-1$, note that for the case of the Lucas sequence of the first kind the sum is
$$
\sum_{i=0}^N U(n+i) \ = \ \frac{1}{\alpha-\beta} \left(\left(\frac{\alpha^{N+1}-1}{\alpha-1}\right)\alpha^n-\left(\frac{\beta^{N+1}-1}{\beta-1}\right)\beta^n\right).
$$
Using $\beta=1/\alpha$ this can be rewritten as
$$
\sum_{i=0}^N U(n+i) \ = \ \frac{1}{\alpha-\beta}\left(\frac{\alpha^{N+1}-1}{\alpha-1}\right)\left(\alpha^n-\alpha^{-n-N}\right).
$$
Taking $N$ to be even this becomes
$$
\left(\frac{\alpha^{N/2+1}-\alpha^{-N/2}}{\alpha-1}\right)\left(\frac{\alpha^{n+N/2}-\beta^{n+N/2}}{\alpha-\beta}\right) \ = \ \left(\frac{\alpha^{N/2+1}-\alpha^{-N/2}}{\alpha-1}\right)U(n+N/2),
$$
and one checks easily that the expression multiplying $U(n+N/2)$ is both an algebraic integer and Galois invariant (it does not change by replacing $\alpha$ with $1/\alpha$) so rational, hence it is an integer. A similar argument works
for Lucas sequences of the second kind.
\end{proof}

We still need to look at the case $\beta=\pm 1$. In this case $\alpha$ is an integer. When $\beta=1$, the congruences become
$$
\left(\frac{\alpha^{N+1}-1}{\alpha-1}\right)\alpha^n \ \equiv \ \eta (N+1)\pmod D,\qquad \alpha^m \ \equiv \ \eta \pmod D.
$$
The number $\gamma$ now becomes
$$
\gamma \ := \ \left(\frac{(N+1)(\alpha-1)}{\alpha^{N+1}-1}\right)^2.
$$
We aim to determine when this is multiplicatively dependent on \(\alpha\). This occurs only for finitely many values of \(N\), because for \(N > 5\), the term \(\alpha^{N+1} - 1\) possesses a primitive prime factor \(p\) such that \(p \equiv 1 \pmod{N+1}\), which does not simplify with the numerator \(N+1\). Consequently, if \(N > \max\{5, |\alpha|\}\), the numbers \(\gamma\) and \(\alpha\) are multiplicatively independent. The case where \(\beta = -1\) follows similarly and is left as an exercise for the reader.
\renewcommand{\thesection}{A}

\subsection{Proof of Theorem \ref{thm:gen-pell-1}}\label{subsec:gen-pell-1}

\begin{proof}
Let
\begin{equation}\label{eqn:Sn-gen-pell-1}
    S_n\ =\ \sum\limits_{i=0}^{n}P_{k}^{k-1}(i).
\end{equation}
Note that the first $k-1$ terms in this sum are 0.
\\ \\ We proceed by induction on $n$, starting at $n=k$ for the base case. Shifting indices on \eqref{eqn:Sn-gen-pell-1} for $n=k$ gives
\[\sum\limits_{i=0}^{2k+1}P_{k}^{k-1}(k+i)\ =\ \sum\limits_{i=k}^{3k+1}P_{k}^{k-1}(i)\ =\ S_{3k+1}.\]
Noting that the first $k-1$ terms of $S_n$ are all zero, we find
\begin{align*}
S_{3k+1}\ &=\ S_{3k}+2P_{k}^{k-1}(3k)+P_{k}^{k-1}(2k)\\
&=\ \underbrace{S_{2k-1}+2P_{k}^{k-1}(2k)+\sum\limits_{i=2k+1}^{3k-1}P_{k}^{k-1}(i)}_{\eqref{eqn-1:gen-pell-1}}+3P_{k}^{k-1}(3k).
\end{align*}
Now, we consider
\begin{equation}
\label{eqn-1:gen-pell-1}
S\ :=\ S_{2k-1}+2P_{k}^{k-1}(2k)+\sum\limits_{i=2k+1}^{3k-1}P_{k}^{k-1}(i).
\end{equation}
Since the first $k-1$ terms of the sum in \eqref{eqn:Sn-gen-pell-1} are zero,
\begin{align*}
    S_{2k-1}\ &=\ \sum\limits_{i=0}^{2k-1}P_{k}^{k-1}(i)\\
    &=\ \sum\limits_{i=k}^{2k-1}P_{k}^{k-1}(i)\\
    &=\ P_{k}^{k-1}(k)+\sum\limits_{i=k+1}^{2k-1}P_{k}^{k-1}(i).
\end{align*}
Then, applying recursion \eqref{rec:gen-pell} to \eqref{eqn-1:gen-pell-1}, we find
\begin{align}
S\ &=\ \sum\limits_{i=k+1}^{2k-1}P_{k}^{k-1}(i)+P_{k}^{k-1}(2k+1)+\sum\limits_{i=2k+1}^{3k-1}P_{k}^{k-1}(i)\nonumber\\
&\ = \ \sum\limits_{i=k+1}^{2k-1}P_{k}^{k-1}(i)+2P_{k}^{k-1}(2k+1)+\sum\limits_{i=2k+2}^{3k-1}P_{k}^{k-1}(i)\nonumber\\
&\ = \ \sum\limits_{i=k+2}^{2k-1}P_{k}^{k-1}(i)+2P_{k}^{k-1}(2k+1)+P_{k}^{k-1}(k+1)+\sum\limits_{i=2k+2}^{3k-1}P_{k}^{k-1}(i)\nonumber\\
&\ \ \vdots \nonumber\\
&\ = \ 2P_{k}^{k-1}(3k-1)+P_{k}^{k-1}(2k-1)\ =\ P_{k}^{k-1}(3k),\end{align}
where the final reduction of $S$ results from alternatively removing terms indexed by the lower bounds of each of the summations and then applying recursion \eqref{rec:gen-pell}. Thus we have $S_{3k+1}=S+3P_{k}^{k-1}(3k)=4P_{k}^{k-1}(3k)$, proving the base case.

Now, by the induction hypothesis we have
\begin{equation}\label{gen-pell-1:ind-hyp}
    \sum\limits_{i=0}^{2k+1}P_{k}^{k-1}(n-1+i)\ =\ 4P_{k}^{k-1}(n+2k-1),
\end{equation}
and by \eqref{sumgenpell} we have
\begin{align}
\nonumber
    \sum\limits_{i=0}^{2k+1} & P_{k}^{k-1}(n-1+i)\ =\ \sum\limits_{i=0}^{n+2k}P_{k}^{k-1}(i)-\sum\limits_{i=0}^{n-2}P_{k}^{k-1}(i)\\
    \label{gen-pell-1:ind-step-1}
    &=\ \frac{1}{2}\left( \sum\limits_{i=0}^{k}P_{k}^{k-1}(n+2k+1-i)-\sum\limits_{i=0}^{k}P_{k}^{k-1}(n-1-i) \right).
\end{align}
Combining \eqref{gen-pell-1:ind-hyp} and \eqref{gen-pell-1:ind-step-1} we get
\begin{equation}
    \sum\limits_{i=0}^{k}P_{k}^{k-1}(n+2k+1-i)
    -\sum\limits_{i=0}^{k}P_{k}^{k-1}(n-1-i)\ =\
    8P_{k}^{k-1}(n+2k-1).
\end{equation}
Furthermore, we have
\begin{align}
\nonumber
\sum\limits_{i=0}^{2k+1}P_{k}^{k-1}(n+i)\
&=\ \frac{1}{2}\left( \sum\limits_{i=0}^{k}P_{k}^{k-1}(n+2k+2-i)-\sum\limits_{i=0}^{k}P_{k}^{k-1}(n-i) \right)\\
\nonumber
&=\ \frac{1}{2}\bigl( P_{k}^{k-1}(n+2k+2)-P_{k}^{k-1}(n+k+1)+8P_{k}^{k-1}(n+2k-1)\\
\nonumber
&\ \ \ +\ P_{k}^{k-1}(n-k-1)-P_{k}^{k-1}(n)\bigl)\\
\nonumber
&=\ \frac{1}{2}\left( 2P_{k}^{k-1}(n+2k+1)+8P_{k}^{k-1}(n+2k-1) -2P_{k}^{k-1}(n-1)\right)\\
\nonumber
&=\ \frac{1}{2}\bigl( 2P_{k}^{k-1}(n+2k+1)-2P_{k}^{k-1}(n+k)+2P_{k}^{k-1}(n+k)\\
\nonumber
&\ \ \ - \ 2P_{k}^{k-1}(n-1)+8P_{k}^{k-1}(n+2k-1) \bigr)\\
\nonumber
&=\ \frac{1}{2}\left( 4P_{k}^{k-1}(n+2k)+4P_{k}^{k-1}(n+k-1)+8P_{k}^{k-1}(n+2k-1) \right)\\
&=\ 4P_{k}^{k-1}(n+2k).\end{align}
We obtain the last equation by using
\begin{equation}
    P_{k}^{k-1}(n+k-1)+2P_{k}^{k-1}(n+2k-1) \ = \ P_{k}^{k-1}(n+2k).
\end{equation}

Now, consider
\begin{equation}
\sum\limits_{i=0}^{2k+1}P_{k}^{k-1}(n+i)\
=\ \frac{1}{2}\left( \underbrace{\sum\limits_{i=0}^{k}P_{k}^{k-1}(n+2k+2-i)}_{s_1\ \eqref{gen-pell-1:s1}}-\underbrace{\sum\limits_{i=0}^{k}P_{k}^{k-1}(n-i)}_{s_2\ \eqref{gen-pell-1:s2}}  \right).
\end{equation}
We have the following explicit formulas for $s_1$ and $s_2$:
\begin{align}
\label{gen-pell-1:s1}
s_1 &= P_{k}^{k-1}(n+2k+2)
      - P_{k}^{k-1}(n+k+1)
      + \sum\limits_{i=0}^{k} P_{k}^{k-1}(n+2k+1-i), \\
\label{gen-pell-1:s2}
s_2 &= P_{k}^{k-1}(n)
      - P_{k}^{k-1}(n-k-1)
      + \sum\limits_{i=0}^{k} P_{k}^{k-1}(n-i-1).
\end{align}

We now note that the RHS of \ref{gen-pell-1:s1} and \ref{gen-pell-1:s2} are particularly amenable to manipulation, and therefore turn our attention towards $ \frac{1}{2} \left(s_1 - s_2 \right)$.

Thus, we have
\begin{align}
\nonumber
\frac{1}{2}(s_1-s_2)\
&=\ \frac{1}{2}\bigl( P_{k}^{k-1}(n+2k+2)-P_{k}^{k-1}(n+k+1)+8P_{k}^{k-1}(n+2k-1)\\
\nonumber
&\ \ \ + \ P_{k}^{k-1}(n-k-1)-P_{k}^{k-1}(n)\bigl)\\
\nonumber
&=\ \frac{1}{2}\left( 2P_{k}^{k-1}(n+2k+1)+8P_{k}^{k-1}(n+2k-1) -2P_{k}^{k-1}(n-1)\right)\\
\nonumber
&=\ \frac{1}{2}\bigl( 2P_{k}^{k-1}(n+2k+1)-2P_{k}^{k-1}(n+k)+2P_{k}^{k-1}(n+k)\\
\nonumber
&\ -2P_{k}^{k-1}(n-1)+8P_{k}^{k-1}(n+2k-1) \bigr)\\
\nonumber
&=\ \frac{1}{2}\left( 4P_{k}^{k-1}(n+2k)+4P_{k}^{k-1}(n+k-1)+8P_{k}^{k-1}(n+2k-1) \right)\\
&=\ 4P_{k}^{k-1}(n+2k).
\end{align}
\end{proof}

\subsection{Proof of Theorem \ref{thm:generalized_fibonacci_2n+1}}\label{subsec:proof_generalized_fibonacci_2n+1}

\begin{proof}
  Let the sum of any $2N+1$ consecutive terms in the $kth$ Fibonacci sequence be $C(N)$ times another integer in the Fibonacci Sequence.
  An argument similar to Theorem \ref{thm:pell-odd} rules out the cases $C(N)>2$. Therefore, we just need to take care of the cases when $C(N)=1,2$.

\ \\
   \textbf{Case 1:} $C(N)=1$. \\ \\ By induction on $r>2k+1$, we have
\begin{equation}\label{eqn:induct-gen-fib}
    \sum\limits_{n=0}^{r}f_k(n) \ < \ f_k(r+2),
\end{equation}
which implies
\begin{equation*}
    \sum\limits_{i=0}^{2N}f_k(n+i) \ \leq \  \sum\limits_{i=0}^{n+2N}f_k(i) \ < \ f_k(n+2N+2).
\end{equation*}
From the definition of the order-$k$ generalized Fibonacci sequence, for $2N>k+1$ we have
\begin{equation*}
f_k(n+2N+1) \ < \ \sum\limits_{i=0}^{2N}f_k(n+i),
\end{equation*} and thus $C(N)\neq 1$.
\\ \\ \textbf{Case 2:} $C(N)=2$. \\ \\ Define $$\lambda_k \ := \ \lim_{n\to \infty} \frac{f_k(n+1)}{f_k(n)}.$$ We now note that \cite[\S11, Theorem 9]{ar} states $\lambda_k+\lambda_{k}^{-k}=2$, which implies that $\lambda_k<2$ and hence for $n > 2k+1$ we have $f_k(n+1)/f_k(n)<2$. Applying \eqref{eqn:induct-gen-fib} thus implies
\begin{equation}
    \sum\limits_{n=0}^{r}f_k(n) \ < \ f_k(r+2) \ < \ 2f_k(r+1).
\end{equation}

Now since, $2N>k+1$, from the definition of the order-$k$ generalized Fibonacci sequence we have
\begin{align}
    f_k(n+2N) \ < \ \sum\limits_{i=0}^{2N-1}f_k(n+i)\nonumber \ \ \implies\ \ 2f_k(n+2N) \ < \ \sum\limits_{i=0}^{2N}f_k(n+i).
\end{align}

Lastly, we have
\begin{align}
    2f_k(n+2N)\ <\ \sum\limits_{i=0}^{2N}f_k(n+i)\ <\
    \sum\limits_{i=0}^{n+2N}f_k(i)\ <\ 2f_k(n+2N+1),
\end{align}
which implies $C(N) \neq 2$, completing the proof.
\end{proof}
\renewcommand{\thesection}{B}
\section{Computational Experiments}
The computational experiments for the paper were carried out in the Wolfram and Python Languages. The GitHub repository can be accessed from  \\ \bburl{https://github.com/navvye/Polymath-Pell-Numbers}.
\renewcommand{\thesection}{C}
\section{Thanks}
We thank the 2023 Polymath Jr REU for creating the opportunity for this work, Stephanie Reyes for numerous comments throughout the research and for numerous technical conversations regarding the paper, and the participants of the 21st International Fibonacci Conference for discussions on this topic related to our presentation. This work was partially supported by NSF Grant DMS2313292.
\ \\

\end{document}